\documentclass{amsart}
\usepackage{amsmath}
\usepackage{amscd}
\usepackage{amssymb}
\usepackage{amsfonts}
\newtheorem{theorem}{Theorem}[section]

\newtheorem{corollary}[theorem]{Corollary}

\theoremstyle{definition}

\theoremstyle{remark}
\newtheorem{remark}[theorem]{Remark}
\numberwithin{equation}{section}
\begin{document}
\title[Differential Harnack inequalities for nonlinear heat equations]
{Differential Harnack inequalities for nonlinear heat equations with
potentials\\ under the Ricci flow}

\author{Jia-Yong Wu}
\address{Department of Mathematics, Shanghai Maritime University,
Haigang Avenue 1550, Shanghai 201306, P. R. China}

\email{jywu81@yahoo.com}

\thanks{This work is partially supported by the NSFC (No.11101267)
and the Science and Technology Program of Shanghai Maritime University
(No. 20120061).}


\dedicatory{}

\date{\today}


\begin{abstract}
We prove several differential Harnack inequalities for positive
solutions to nonlinear backward heat equations with different potentials coupled
with the Ricci flow. We also derive an interpolated Harnack inequality for the
nonlinear heat equation under the $\varepsilon$-Ricci flow on a closed surface.
These new Harnack inequalities extend the previous differential Harnack
inequalities for linear heat equations with potentials under the Ricci flow.
\end{abstract}
\maketitle

% ------------------------------------------------------------------------

\section{Introduction and main results}\label{sec1}

\subsection{Background}
The study of differential Harnack estimates for parabolic equations
originated with the work of P. Li and S.-T. Yau~\cite{[Li-Yau]},
who first proved a gradient estimate for the heat equation via the
maximum principle (though a precursory form of their estimate
appeared in \cite{[ArBe]}). Using their gradient estimate, the same
authors derived a classical Harnack inequality by integrating
the gradient estimate along space-time paths. This result was generalized
to Harnack inequalities for some nonlinear heat-type equations in \cite{[Yau1]} and
for some non-self-adjoint evolution equations in~\cite{[Yau2]}. Recently, J. Li and
X. Xu \cite{[LiXu]} gave sharper local estimates than previous results for the heat equation
on Riemannian manifolds with Ricci curvature bounded below. Surprisingly,
R. Hamilton employed similar techniques to obtain Harnack inequalities for the
Ricci flow~\cite{[Ham2]}, and the mean curvature flow~\cite{[Ham4]}. In
dimension two, a differential Harnack estimate for the positive scalar curvature
was proved in \cite{[Ham1]}, and then extended by B. Chow ~\cite{[Chow0]} when the
scalar curvature changes sign. Similar techniques were used to obtain the Harnack
inequalities for the Gauss curvature flow~\cite{[Chow1]} and the Yamabe flow
~\cite{[Chow2]}. H.-D. Cao~\cite{[Cao]} proved a Harnack inequality for the
K\"ahler-Ricci flow. B. Andrews~\cite{[Andrews]} derived several Harnack inequalities
for general curvature flows of hypersurfaces. Chow and Hamilton \cite{[ChHam]}
gave extensions of the Li-Yau Harnack inequality, which they called constrained
and linear Harnack inequalities. For more detailed discussion, we refer the
interested reader to Chapter 10 of~\cite{[CLN]}.

Hamilton~\cite{[Ham3]} also generalized the Li-Yau Harnack inequality to a matrix
Harnack form on a class of Riemannian manifolds with nonnegative sectional curvature.
This result was extended to the constrained matrix Harnack inequalities in
\cite{[ChHam]}. H.-D. Cao and L. Ni ~\cite{[CaoNi]} proved a matrix Harnack
estimate for the heat equation on K\"ahler  manifolds. Chow and Ni~\cite{[Ni1]} proved
a matrix Harnack estimate for K\"ahler-Ricci flow using interpolation techniques
from~\cite{[Chow3]}.

In another direction, differential Harnack inequalities for (backward) heat-type
equations coupled with the Ricci flow have become an important object, which
can be traced back to \cite{[Ham1]}. This subject was further explored by
Chow~\cite{[Chow3]}, Chow and Hamilton \cite{[ChHam]}, Chow and D. Knopf \cite{[ChKn]},
and H.-B. Cheng \cite{[HBCheng]}, etc. Perhaps the most spectacular result is
G. Perelman's~\cite{[Perelman]} differential Harnack inequality for the
fundamental solution to the backward heat equation coupled with the Ricci flow
without any curvature assumption. Perelman's Harnack inequality has many
important applications (it is essential in proving pseudolocality theorems),
and it has been extended by X. Cao \cite{[Caox]} and independently by
S.-L. Kuang and Qi S. Zhang~\cite{[KuZh]}. Those authors proved a differential
Harnack inequality for all positive solutions to the backward heat equation
under the Ricci flow on closed manifolds with nonnegative scalar curvature.
X. Cao and Qi S. Zhang~\cite{[CaoZhq]} have established Gaussian upper and
lower bounds for the fundamental solution to the backward heat equation
under the Ricci flow.

On the subject of differential Harnack inequalities for the linear heat equation
coupled with the Ricci flow, there have been many important contributions; see,
for example,~\cite{[BaCaoPu]},~\cite{[CaoxHa]},~\cite{[ChTaYu]},~\cite{CCG3},
~\cite{[Guen]},~\cite{[Liu]},~\cite{[WuZheng]} and~\cite{[Zhang]}.

In recent years there has been increasing interest in the study of the nonlinear
heat-type equations coupled with the Ricci flow. A nice example of a nonlinear
heat equation, introduced by L. Ma~\cite{[Ma]}, is
\begin{equation}\label{form}
\frac{\partial}{\partial t}f=\Delta f-af\ln f-bf,
\end{equation}
where $a$ and $b$ are real constants. Ma first proved a local gradient estimate
for positive solutions to the corresponding elliptic equation
\begin{equation}\label{ellform}
\Delta f-af\ln f-bf=0
\end{equation}
on a complete manifold with a fixed metric. Indeed, F. R. K. Chung
and S.-T. Yau ~\cite{[ChuYau]} observed that equation\eqref{ellform}
is linked with the gross logarithmic Sobolev inequality. They also
established a logarithmic Harnack inequality for this equation when
$a<0$. Y. Yang~\cite{[Yang]} derived local gradient estimates for positive solutions
to \eqref{form} on a complete manifold with a fixed metric; see also \cite{[ChCh]},
\cite{[HuangMa]}, \cite{[Wu1]} and \cite{[Wu2]}. Yang's result has been generalized
by L. Ma~\cite{[Ma2], [Ma3]}, who obtained Hamilton and new Li-Yau type gradient
estimates for the nonlinear heat equation \eqref{form}, and also by S.-Y. Hsu~\cite{[Hsu]},
who proved local gradient estimates for the nonlinear heat equation \eqref{form} under the Ricci
flow, similar to the gradient estimates of~\cite{[Yang]} for the fixed metric case

We remind the reader that equations \eqref{form} and \eqref{ellform} often appear
in geometric evolution equations, and are also closely related to the gradient
Ricci solitons. See, for example, \cite{[CaoZhang]} and \cite{[Ma]} for nice
explanations on this subject.

Very recently, X. Cao and Z. Zhang~\cite{[CaoZhang]} used the argument from
\cite{[CaoxHa]} to prove an interesting differential Harnack inequality for positive
solutions to the forward nonlinear heat equation
\begin{equation}\label{form3}
\frac{\partial}{\partial t}f=\Delta f-f\ln f+Rf
\end{equation}
coupled with the Ricci flow equation
\begin{equation}\label{RF}
\frac{\partial}{\partial t}g_{ij}=-2R_{ij}
\end{equation}
on a closed manifold. Here the symbol $\Delta$, $R$ and $R_{ij}$ are
the Laplacian, scalar curvature and Ricci curvature of the metric
$g(t)$ moving under the Ricci flow.

\subsection{Main results}
In this paper, we will be concerned with general time-dependent
nonlinear backward heat equations of the type \eqref{form} with
different potentials on closed manifolds under the Ricci
flow.

Before studying nonlinear backward heat equations, we
first study the nonlinear forward heat equation \eqref{form3} with the
metric evolving under the Ricci flow. Suppose $(M,g(t))$,
$t\in[0,T)$, is a solution to the $\varepsilon$-Ricci flow ($\varepsilon\geq0$)
\begin{equation}\label{psRF}
\frac{\partial}{\partial t}g_{ij}=-\varepsilon Rg_{ij}
\end{equation}
on a closed surface. Let $f$ be a positive solution
to the nonlinear forward heat equation with potential $\varepsilon R$, that is,
\begin{equation}\label{foreq1}
\frac{\partial}{\partial t}f=\Delta f-f\ln f+\varepsilon Rf.
\end{equation}
In this case, we can derive a
new differential interpolated Harnack inequality, which is
originated with B. Chow~\cite{[Chow3]}.

\begin{theorem}\label{interposur}
Let $(M,g(t))$, $t\in[0,T)$, be a solution to the
$\varepsilon$-Ricci flow \eqref{psRF}
on a closed surface with $R>0$. Let $f$ be a positive solution to
the nonlinear heat equation \eqref{foreq1},
$u=-\ln f$ and $H_\varepsilon=\Delta u-\varepsilon R$.
Then for all time $t\in(0,T)$,
\[
H_\varepsilon\leq \frac 1t,
\]
i.e.,
\[
\frac{\partial}{\partial t}\ln f-|\nabla\ln f|^2+\ln f+\frac
1t=\Delta\ln f+\varepsilon R+\frac 1t\geq 0.
\]
\end{theorem}

In Theorem \ref{interposur}, if we take $\varepsilon=0$, we can get
the following differential Harnack inequality for the nonlinear heat
equation on closed surfaces with a fixed metric:
\begin{corollary}\label{interfix}
If $f: M\times[0,T)\to\mathbb{R}$, is a positive solution to the
nonlinear heat equation
\[
\frac{\partial}{\partial t}f=\Delta f-f\ln f
\]
on a closed surface $(M,g)$ with
$R>0$, then for all time $t\in(0,T)$,
\[
\frac{\partial}{\partial t}\ln f-|\nabla\ln f|^2+\ln f+\frac
1t=\Delta\ln f+\frac 1t\geq 0.
\]
\end{corollary}

If we take $\varepsilon=1$ in Theorem \ref{interposur}, we get:
\begin{corollary}\label{formainsur}
Let $(M,g(t))$, $t\in[0,T)$, be a solution to the Ricci flow on a
closed surface with $R>0$. If $f$ is a positive solution to the
nonlinear heat equation \eqref{form3}, then for all time
$t\in(0,T)$,
\[
\frac{\partial}{\partial t}\ln f-|\nabla\ln f|^2+\ln f+\frac
1t=\Delta\ln f+R+\frac 1t\geq 0.
\]
\end{corollary}

\begin{remark}
X. Cao and Z. Zhang~\cite{[CaoZhang]} have proved a differential Harnack
inequality for equation \eqref{form3} under the Ricci flow on any
dimensional manifold. However, on a closed surface, the result of
Corollary \ref{formainsur} is better than theirs.
\end{remark}

\begin{remark}
Interestingly, Theorem \ref{interposur} is a nonlinear interpolated
Harnack inequality which links Corollary \ref{interfix} to Corollary
\ref{formainsur}.
\end{remark}

\vspace{1em}

Secondly, we now consider differential Harnack inequalities for positive
solutions to the nonlinear backward heat equation with potential
$2R$, that is,
\begin{equation}\label{eq1}
\frac{\partial}{\partial t}f=-\Delta f+f\ln f+2Rf
\end{equation}
under the Ricci flow. X. Cao and Z. Zhang \cite{[CaoZhang]} made nice explanations that
the nonlinear forward heat equation \eqref{form3} is closely related to expanding gradient
Ricci solitons. Analogously to the argument of Cao and Zhang, our consideration of the
Equation \eqref{eq1} is motivated by \emph{shrinking} gradient Ricci solitons proposed in
\cite{[Ham2]}. Recall that a shrinking gradient Ricci soliton $(M,g)$ is defined
by the form (see \cite{[CLN]})
\begin{equation}\label{soliton}
R_{ij}+\nabla_i\nabla_jw=cg_{ij},
\end{equation}
where $w$ is some Ricci soliton potential and $c$ is a positive
constant. Taking the trace of both sides of \eqref{soliton} yields
\begin{equation}\label{soliTr}
R+\Delta w=\mathrm{const}.
\end{equation}
Using the contracted Bianchi identity, we can easily deduce that
\begin{equation}\label{soliBich}
R-2cw+|\nabla w|^2=-\mathrm{const}.
\end{equation}
From \eqref{soliTr} and \eqref{soliBich}, we get
\begin{equation}\label{comb}
2|\nabla w|^2=-\Delta w+|\nabla w|^2+2cw-2R.
\end{equation}

Recall that the Ricci flow solution for a complete gradient Ricci soliton
(\cite{[CLN]}, Theorem 4.1) is the pullback of $g$ under  $\varphi(t)$,
up to a scale factor $c(t)$:
\[
g(t)=c(t)\cdot\varphi(t)^{*}g,
\]
where $c(t):=-2ct+1>0$ and $\varphi(t)$ is the 1-parameter family of
diffeomorphisms generated by
\[
\frac{1}{c(t)}\nabla_g w.
\]
Then the corresponding Ricci soliton potential $\varphi(t)^{*}w$ satisfies
\[
\frac{\partial}{\partial t} \varphi(t)^{*}w=\left|\nabla\varphi(t)^{*}w\right|^2.
\]
Note that along the Ricci flow, \eqref{comb} becomes
\[
2|\nabla\varphi(t)^{*}w|^2=-\Delta\varphi(t)^{*}w+|\nabla\varphi(t)^{*}w|^2
+\frac{2c}{c(t)}\cdot\varphi(t)^{*}w-2R.
\]
Hence the evolution equation for the Ricci soliton potential $\varphi(t)^{*}w$ is
\[
2\frac{\partial\varphi(t)^{*}w}{\partial t}
=-\Delta\varphi(t)^{*}w+|\nabla\varphi(t)^{*}w|^2
+\frac{2c}{c(t)}\cdot\varphi(t)^{*}w-2R.
\]
If we let $\varphi(t)^{*}w=-\ln \tilde{f}$, this equation becomes
\begin{equation}\label{comb3}
2\frac{\partial\tilde{f}}{\partial t}=-\Delta\tilde{f}+2R\tilde{f}
+\frac{2c}{c(t)}\cdot \tilde{f}\ln \tilde{f}.
\end{equation}
Notice that \eqref{eq1} and \eqref{comb3} are closely related and only differ by
the time scaling and their last terms.

For the nonlinear backward heat equation \eqref{eq1} under the Ricci flow, we have:
\begin{theorem}\label{main1}
Let $(M, g(t))$, $t\in[0,T]$, be a solution to the Ricci flow on a
closed manifold. Let $f$ be a positive solution to the nonlinear
backward heat equation \eqref{eq1}, $u=-\ln f$, $\tau=T-t$ and
\begin{equation}\label{harna1}
H=2\Delta u-|\nabla u|^2+2R-2\frac{n}{\tau}.
\end{equation}
Then for all time $t\in[0,T)$,
\[
H\leq \frac {n}{2}.
\]
\end{theorem}

\begin{remark}\label{equi}
We can easily see that $H\leq \frac n2$ is equivalent to
\[
\frac{|\nabla f|^2}{f^2}-2\left(\frac{f_\tau}{f}+\ln f+
R\right)\leq 2\frac n\tau+\frac n2.
\]
In \cite{[Yang]} (see also \cite{[Wu2]}), the classical
Li-Yau gradient estimate for positive solutions to the nonlinear
heat equation \eqref{form} is
\[
\frac{|\nabla f|^2}{f^2}-2\left(\frac{f_t}{f}+a\ln f+b\right)\leq
2\frac nt+na
\]
on manifolds with a fixed metric satisfying nonnegative Ricci
curvature. Hence our Harnack inequality is similar to the classical
Li-Yau gradient estimate for the nonlinear heat equation
\eqref{form}.
\end{remark}

If we assume instead that our solution to the Ricci flow is defined
for $t\in[0,T)$ (where $T<\infty$ is the blow-up time) and is of
type I, meaning that
\begin{equation}\label{type1}
|Rm|\leq\frac{d_0}{T-t}
\end{equation}
for some constant $d_0$, then we can show this:
\begin{theorem}\label{main2}
Let $(M, g(t))$, $t\in[0,T)$ (where $T<\infty$ is the blow-up time) be a
solution to the Ricci flow on a closed manifold of dimension $n$, and assume
that $g$ is of type I, that is, it satisfies \eqref{type1}, for some constant
$d_0$. Let $f$ be a positive solution to the nonlinear backward heat equation
\eqref{eq1}, $u=-\ln f$, $\tau=T-t$ and
\[
H=2\Delta u-|\nabla u|^2+2R-d\frac{n}{\tau},
\]
where $d=d(d_0,n)\geq 2$ is some constant such that $H(\tau)<0$ for
small $\tau$. Then, for all time $t\in[0,T)$,
\[
H\leq \frac n2.
\]
\end{theorem}

\vspace{1em}

Thirdly, we consider the nonlinear backward heat equation
\begin{equation}\label{conjeq1}
\frac{\partial}{\partial t}f=-\Delta f+f\ln f+Rf
\end{equation}
under the Ricci flow. This equation is very similar to equation
\eqref{eq1} and only differs by the last potential. We also find
that \eqref{conjeq1} can be regarded as
the extension of the linear backward heat equation considered
by X. Cao (see Theorem 1.3 in ~\cite{[Caox]}), and S.-L. Kuang and
Qi S. Zhang (see Theorem 2.1 in~\cite{[KuZh]}). In fact, we only have
the additional term : $f\ln f$ in the linear backward heat equation.
For this system, we prove:
\begin{theorem}\label{main3}
Let $(M, g(t))$, $t\in[0,T]$, be a solution to the Ricci flow on a
closed manifold with nonnegative scalar curvature. Let $f$ be a
positive solution to the nonlinear backward heat equation
\eqref{conjeq1}, $u=-\ln f$, $\tau=T-t$ and
\[
H=2\Delta u-|\nabla u|^2+R-2\frac n\tau.
\]
Then, for all time $t\in[0,T)$,
\[
H\leq \frac n4.
\]
\end{theorem}

By modifying the Harnack quantity of Theorem
\ref{main3}, we can deduce the following differential Harnack
inequalities \emph{without} assuming the nonnegativity of $R$:
\begin{theorem}\label{main3b}
Let $(M, g(t))$, $t\in[0,T]$, be a solution to the Ricci flow on a closed
manifold of dimension $n$. Let $f$ be a positive solution to the nonlinear
backward heat equation \eqref{conjeq1}, $v=-\ln f-\frac n2\ln(4\pi
\tau)$ , $\tau=T-t$, and
\[
P=2\Delta v-|\nabla v|^2+R-3\frac n\tau.
\]
Then, for all time $t\in[\frac T2,T)$,
\[
P\leq \frac n4.
\]
\end{theorem}

\begin{remark}
Theorems \ref{main1}-\ref{main3b} extend to the nonlinear case Theorems 1.1¨C1.3
and 3.6 of \cite{[Caox]} and Theorem 2.1 of \cite{[KuZh]}.
\end{remark}

The proof of all our theorems nearly follows from the arguments of X.
Cao~\cite{[Caox]}, X. Cao and R. Hamilton~\cite{[CaoxHa]}, X. Cao
and Z. Zhang~\cite{[CaoZhang]}, and S.-L. Kuang and Qi S.
Zhang~\cite{[KuZh]}, where computations of evolution equations and
the maximum principle for parabolic equations are employed. The major
differences are that one of our results gives an interpolation
Harnack inequality for a nonlinear forward heat equation
along the $\varepsilon$-Ricci flow on a closed surface,
and the others provide differential Harnack estimates for various
\textbf{nonlinear backward} heat equations under the Ricci flow.

One interesting feature of this paper is that our differential
Harnack inequalities are not only like the Perelman's Harnack
inequalities, but also similar to the classical Li-Yau Harnack
inequalities for the corresponding nonlinear heat equation (see Remark \ref{equi}
above). Another feature of this paper is that our Harnack quantities
of nonlinear backward heat equations are nearly the same as those of linear
backward heat equations considered by X. Cao~\cite{[Caox]}, and S.-L. Kuang
and Qi S. Zhang~\cite{[KuZh]}. Due to the fact that Ricci soliton potentials are linked
with some nonlinear backward heat equations, we expect that our differential
Harnack inequalities will be useful in understanding the Ricci solitons.

The rest of this paper is organized as follows: In Sect.~\ref{sec4},
we will prove a new differential interpolated Harnack inequality
on a surface, i.e., Theorem \ref{interposur}. In Sect.~\ref{sec2},
we firstly derive differential Harnack inequalities for positive
solutions to the nonlinear backward heat equation with potential $2R$
under the Ricci flow (Theorems \ref{main1} and \ref{main2}). Then
a classical integral version of the Harnack inequality will be proved
(Theorem \ref{classmain1}). In the latter part of this section, we
will establish Harnack inequalities for another nonlinear backward
heat equation with potential $R$ under the Ricci flow
(Theorem \ref{main3}) as well as its classical Harnack version
(Theorem \ref{classmain3}). By modifying the Harnack quantity of
Theorem \ref{main3}, we can prove another differential Harnack
inequalities without the nonnegative assumption of scalar curvature
(Theorem \ref{main3b}). Finally, in Sect.~\ref{sec5}, we
will prove gradient estimates for positive and bounded solutions to the
nonlinear (including backward) heat equation without potentials
under the Ricci flow, i.e., Theorems \ref{heatmain1} and \ref{heatmain2}.

% --------------------------------------------------------------------------
% ------------------------------------------------------------------------

\section{Nonlinear heat equation with potentials}\label{sec4}
In this section, we will prove a differential interpolated Harnack inequality
for positive solutions to nonlinear forward heat equations with potentials
coupled with the $\varepsilon$-Ricci flow on a closed surface.

Let $f$ be a positive solution to the nonlinear forward heat equation \eqref{foreq1}. By
the maximum principle, we conclude that the solution will remain positive along
the Ricci flow when scalar curvature is positive. If we let
\[
u=-\ln f,
\]
then $u$ satisfies the equation
\[
\frac{\partial}{\partial t}u=\Delta u-|\nabla u|^2-\varepsilon R-u.
\]

\begin{proof}[Proof of Theorem \ref{interposur}]
The proof involves a direct computation and the parabolic
maximum principle. Let $f$ and $u$ be defined as above.
Under the $\varepsilon$-Ricci flow \eqref{psRF} on a closed
surface, we have that
\[
\frac{\partial R}{\partial t}=\varepsilon(\Delta R+R^2)
\]
and
\[
\frac{\partial}{\partial t}(\Delta)=\varepsilon R\Delta,
\]
where the Laplacian $\Delta$ is acting on functions. Define the Harnack quantity
\begin{equation}\label{Harquant}
H_\varepsilon=\Delta u-\varepsilon R.
\end{equation}
Using the evolution equations above, we first compute that
\begin{equation*}
\begin{aligned}
\frac{\partial}{\partial t}H_\varepsilon&=\Delta\left(\frac{\partial}{\partial
t}u\right)+\left(\frac{\partial}{\partial t}\Delta\right)u
-\varepsilon\frac{\partial R}{\partial t}\\
&=\Delta\left(\Delta u-|\nabla u|^2-\varepsilon
R-u\right)+\varepsilon R\Delta u-\varepsilon\frac{\partial
R}{\partial t}\\
&=\Delta H_\varepsilon-\Delta|\nabla u|^2-\Delta u
+\varepsilon RH_\varepsilon+\varepsilon^2
R^2-\varepsilon\frac{\partial R}{\partial t}
\end{aligned}
\end{equation*}
Since
\[
\Delta|\nabla u|^2=2|\nabla\nabla u|^2+2\nabla\Delta u\cdot\nabla u+R|\nabla u|^2
\]
on a two-dimensional surface, we then have
\begin{equation*}
\begin{aligned}
\frac{\partial}{\partial t}H_\varepsilon
&=\Delta H_\varepsilon-2|\nabla\nabla u|^2
-2\nabla\Delta u\cdot\nabla u-R|\nabla u|^2\\
&\quad+\varepsilon RH_\varepsilon+\varepsilon^2
R^2-\varepsilon\frac{\partial R}{\partial t}-\Delta u\\
&=\Delta H_\varepsilon-2|\nabla\nabla u|^2
-2\nabla H_\varepsilon\cdot\nabla u-2\varepsilon\nabla R\cdot\nabla u\\
&\quad-R|\nabla u|^2+\varepsilon RH_\varepsilon+\varepsilon^2
R^2-\varepsilon\frac{\partial R}{\partial t}-\Delta u\\
&=\Delta H_\varepsilon-2\left|\nabla_i\nabla_ju-\frac\varepsilon2Rg_{ij}\right|^2
-2\varepsilon R\Delta u-2\nabla H_\varepsilon\cdot\nabla u\\
&\quad-2\varepsilon\nabla R\cdot\nabla u-R|\nabla u|^2
+\varepsilon RH_\varepsilon+2\varepsilon^2R^2
-\varepsilon\frac{\partial R}{\partial t}-\Delta u.
\end{aligned}
\end{equation*}
Since $\Delta u=H_\varepsilon+\varepsilon R$ by \eqref{Harquant}
these equalities become
\begin{equation*}
\begin{aligned}
\frac{\partial}{\partial t}H_\varepsilon
&=\Delta H_\varepsilon-2\left|\nabla_i\nabla_ju-\frac\varepsilon2Rg_{ij}\right|^2
-\varepsilon RH_\varepsilon-2\nabla H_\varepsilon\cdot\nabla u\\
&\quad-2\varepsilon\nabla R\cdot\nabla u-R|\nabla u|^2
-\varepsilon\frac{\partial R}{\partial t}-\Delta u.
\end{aligned}
\end{equation*}
Rearranging terms yields
\begin{equation}
\begin{aligned}\label{kkevoposur}
\frac{\partial}{\partial t}H_\varepsilon&=\Delta H_\varepsilon-
2\left|\nabla_i\nabla_ju-\frac\varepsilon2Rg_{ij}\right|^2
-2\nabla H_\varepsilon\cdot\nabla u-\varepsilon RH_\varepsilon\\
&\quad-R\left|\nabla u+\varepsilon\nabla\ln R\right|^2 -\varepsilon
R\left(\frac{\partial\ln R}{\partial t}
-\varepsilon|\nabla\ln R|^2\right)-\Delta u\\
&\leq\Delta H_\varepsilon-H_\varepsilon^2-2\nabla
H_\varepsilon\cdot\nabla u-(\varepsilon R+1)H_\varepsilon+\frac
\varepsilon tR-\varepsilon R.
\end{aligned}
\end{equation}
The reason for this last inequality is that the trace Harnack inequality for the
$\varepsilon$-Ricci flow on a closed surface proved in \cite{[Chow3]}
(see also Lemma 2.1 in \cite{[WuZheng]}) states that
\[
\frac{\partial\ln R}{\partial t}-\varepsilon|\nabla\ln
R|^2=\varepsilon(\Delta\ln R+R)\geq -\frac 1t,
\]
since $g(t)$ has positive scalar curvature. Besides this, we also
used \eqref{Harquant} and the elementary inequality
\[
\left|\nabla_i\nabla_ju-\frac \varepsilon 2Rg_{ij}\right|^2\geq
\frac 12(\Delta u-\varepsilon R)^2=\frac 12H_\varepsilon^2.
\]
Adding $-\frac 1t$ to $H_\varepsilon$ in \eqref{kkevoposur} yields
\begin{equation}
\begin{aligned}\label{kkepolra}
\frac{\partial}{\partial t}\left(H_\varepsilon-\frac
1t\right)&\leq\Delta \left(H_\varepsilon-\frac
1t\right)-2\nabla\left(H_\varepsilon-\frac 1t\right)\cdot\nabla u\\
&\quad-\left(H_\varepsilon+\frac 1t\right)\left(H_\varepsilon-\frac 1t\right)
-(\varepsilon R+1)\left(H_\varepsilon-\frac
1t\right)-\frac 1t-\varepsilon R.
\end{aligned}
\end{equation}
Clearly, for $t$ small enough we have $H_\varepsilon-1/t<0$.
Since $R>0$, applying the maximum principle to the evolution
formula \eqref{kkepolra} we conclude $H_\varepsilon-1/t\leq 0$
for all time $t$, and the proof of this theorem is completed.
\end{proof}

We remark that Theorem \ref{interposur} can be regarded as a
nonlinear version of an interpolated Harnack inequality proved
by B. Chow:
\begin{theorem}[B. Chow~\cite{[Chow3]}]\label{cor1s}
Let $(M,g(t))$ be a solution to the $\varepsilon$-Ricci flow
\eqref{psRF} on a closed surface with $R>0$. If $f$ is a positive
solution to
\[
\frac{\partial}{\partial t}f=\Delta f+\varepsilon Rf,
\]
then
\[
\frac{\partial }{\partial t}\ln f-|\nabla\ln f|^2+\frac 1t=\Delta
\ln f+\varepsilon R+\frac 1t\geq0.
\]
\end{theorem}

% --------------------------------------------------------------------------
% --------------------------------------------------------------------------

\section{Nonlinear backward heat equation with potentials}\label{sec2}

We next study several differential Harnack inequalities for positive solutions to the
nonlinear backward heat equation under the Ricci flow, proving Theorems \ref{main1}, \ref{main2},
\ref{main3}, and \ref{main3b} from the Introduction. The first two of these theorems deal with the
case where the potential equals $2R$, and the last two with the potential $R$. The
proofs are largely based on the maximum principle.

\vspace{0.5em}

\begin{itemize}
\item \emph{The potential equals to $2R$}
\end{itemize}

Theorems \ref{main1} and \ref{main2} deal with differential Harnack inequalities for
positive solutions to the equation
\[
\frac{\partial}{\partial t}f=-\Delta f+f\ln f+2Rf
\]
under the Ricci flow.  We follow the trick used to prove Theorem 1.1 in \cite{[CaoZhang]}
to simplify a tedious calculation of the evolution equations. Also
the evolution equation of $u$ in this case is very similar to what is considered in
\cite{[Caox]}. So we can borrow Cao's computation for the very general setting there to
simplify our calculation. The only difference is that we have extra terms coming
from the time derivative $\frac{\partial}{\partial \tau}u$.
\begin{proof}[Proof of Theorems \ref{main1}]
As before, it is easy to compute that $u$ satisfies the following
equation
\begin{equation}\label{eq2}
\frac{\partial}{\partial \tau}u=\Delta u-|\nabla u|^2+2R-u.
\end{equation}
Let
\begin{equation}\label{def2}
H=2\Delta u-|\nabla u|^2+2R-2\frac{n}{\tau}.
\end{equation}
Comparing with the equation (2.4) in \cite{[Caox]}, using \eqref{eq2}, we have
\begin{equation}
\begin{aligned}\label{evol1}
\frac{\partial}{\partial \tau}H&=\Delta H-2\nabla H\cdot \nabla
u-\frac{2}{\tau}H-\frac{2}{\tau}|\nabla u|^2-2|Rc|^2\\
&\quad-2\left|\nabla_i\nabla_ju+R_{ij}-\frac{1}{\tau}g_{ij}\right|^2
-2(\Delta u-|\nabla u|^2)\\
&\leq\Delta H-2\nabla H\cdot\nabla
u-\frac{2}{\tau}H-\frac{2}{\tau}|\nabla u|^2-\frac 2nR^2\\
&\quad-\frac 2n\left(\Delta u+R-\frac{n}{\tau}\right)^2-2(\Delta
u-|\nabla u|^2),
\end{aligned}
\end{equation}
where we used the elementary inequality
\[
\left|\nabla_i\nabla_ju-R_{ij}-\frac1\tau g_{ij}\right|^2\geq \frac 1n
\left(\Delta u-R-\frac n\tau\right)^2.
\]
By the definition of $H$ in \eqref{def2}, we also note that
\[
-2(\Delta u-|\nabla u|^2)=-2H+2\left(\Delta
u+R-\frac{n}{\tau}\right)+2R-\frac{2n}{\tau}.
\]
Plugging this into \eqref{evol1} yields
\begin{equation*}
\begin{aligned}
\frac{\partial}{\partial\tau}H&\leq\Delta H-2\nabla H\cdot\nabla
u-\left(\frac{2}{\tau}+2\right)H-\frac{2}{\tau}|\nabla u|^2-\frac
2nR^2\\
&\quad-\frac 2n\left(\Delta u+R-\frac{n}{\tau}-\frac n2\right)^2
+\frac n2+2R-\frac{2n}{\tau}\\
&=\Delta H-2\nabla H\cdot \nabla u
-\left(\frac{2}{\tau}+2\right)H-\frac{2}{\tau}|\nabla u|^2\\
&\quad-\frac 2n\left(\Delta u+R-\frac{n}{\tau}-\frac
n2\right)^2-\frac 2n\left(R-\frac n2\right)^2-\frac{2n}{\tau}+n.
\end{aligned}
\end{equation*}
Adding $-\frac n2$ to $H$, we have
\begin{equation}
\begin{aligned}\label{evol3}
\frac{\partial}{\partial \tau}\left(H-\frac n2\right)&\leq\Delta
\left(H-\frac n2\right)-2\nabla\left(H-\frac n2\right)\cdot\nabla
u-\left(\frac{2}{\tau}+2\right)\left(H-\frac n2\right)\\
&\quad-\frac{2}{\tau}|\nabla u|^2-\frac 2n\left(\Delta
u+R-\frac{n}{\tau}-\frac n2\right)^2-\frac 2n\left(R-\frac
n2\right)^2-\frac{3n}{\tau}.
\end{aligned}
\end{equation}
If $\tau$ is small enough, we can easily see that $H-\frac n2<0$. Then
applying the maximum principle to the evolution equation
\eqref{evol3} yields
\[
H-\frac n2\leq 0
\]
for all time $\tau$, hence for all $t\in[0,T)$.  This finishes the
proof of Theorem \ref{main1}.
\end{proof}
An easy modification of the preceding proof, using \eqref{type1} to
ensure that we can apply the maximum principle as $\tau\to0$,
verifies Theorem \ref{main2}. We omit the details.

\begin{remark}
Theorem \ref{main1} is also true on a complete noncompact Riemannian
manifolds, as long as we can apply the maximum principle.
\end{remark}

From Theorem \ref{main1}, we can derive a classical Harnack
inequality by integrating along a space-time path.
\begin{theorem}\label{classmain1}
Let $(M,g(t))$, $t\in[0,T]$, be a solution to the Ricci flow on a closed
manifold of dimension $n$. Let $f$ be a positive solution to the nonlinear
backward heat equation \eqref{eq1}. Assume that $(x_1,t_1)$ and
$(x_2,t_2)$, $0\leq t_1<t_2<T$, are two points in $M\times[0,T)$.
Then we have
\[
e^{t_2}\cdot\ln f(x_2,t_2)-e^{t_1}\cdot\ln f(x_1,t_1)\leq\frac
12\int^{t_2}_{t_1}e^{T-t}\left(|\dot{\gamma}|^2+2R+\frac
n2+\frac {2n}{T-t}\right)dt,
\]
where $\gamma$ is any space-time path joining $(x_1,t_1)$ and $(x_2,
t_2)$.
\end{theorem}

\begin{proof}
This is similar to Theorem 2.3 in \cite{[Caox]}. we include the proof for
completeness. Consider the solutions to
\[
\frac{\partial}{\partial \tau}u=\Delta u-|\nabla u|^2+2R-u.
\]
Combining this with
\[
H-\frac n2=2\Delta u-|\nabla u|^2+2R-2\frac{n}{\tau}-\frac n2\leq
0,
\]
we have
\[
2\frac{\partial}{\partial \tau}u+|\nabla u|^2-2R-2\frac
n\tau+2u-\frac n2\leq 0.
\]
If $\gamma(x,t)$ is a space-time path joining $(x_2,\tau_2)$ and
$(x_1,\tau_1)$, with $\tau_1>\tau_2>0$. We have along $\gamma$
\begin{equation*}
\begin{aligned}
\frac{du}{d\tau}&=\frac{\partial u}{\partial \tau}+\nabla u\cdot\gamma\\
&\leq-\frac12|\nabla u|^2+R+\frac n\tau-u+\frac n4+\nabla
u\cdot\gamma\\
&\leq\frac12\left(|\dot{\gamma}|^2+2R+\frac n2\right)+\frac n\tau-u,
\end{aligned}
\end{equation*}
where in the last step above we used the inequality
\[
-\frac12|\nabla u|^2+\nabla u\cdot\gamma-
\frac12|\dot{\gamma}|^2\leq 0.
\]
Rearranging terms yields
\[
\frac{d}{d\tau}\left(e^{\tau}\cdot
u\right)\leq\frac{e^{\tau}}{2}\left(|\dot{\gamma}|^2+2R+\frac
n2+\frac{2n}{\tau}\right).
\]
Integrating this inequality we obtain
\[
e^{\tau_1}\cdot u(x_1,\tau_1)-e^{\tau_2}\cdot
u(x_2,\tau_2)\leq\frac
12\inf_{\gamma}\int^{\tau_1}_{\tau_2}e^{\tau}\left(|\dot{\gamma}|^2+2R+\frac
n2+\frac{2n}{\tau}\right)d\tau,
\]
which can be rewritten as
\[
e^{t_1}\cdot u(x_1,t_1)-e^{t_2}\cdot u(x_2,t_2)\leq\frac
12\inf_{\gamma}\int^{t_2}_{t_1}e^{T-t}\left(|\dot{\gamma}|^2+2R+\frac
n2+\frac{2n}{T-t}\right)dt.
\]
Note that $u=-\ln f$. Hence the desired classical Harnack inequality follows.
\end{proof}

% --------------------------------------------------------------------------

\vspace{0.5em}

\begin{itemize}
\item \emph{The potential equals to $R$}
\end{itemize}

We now turn to the equation with potential $R$:
\[
\frac{\partial}{\partial t}f=-\Delta f+f\ln f+Rf.
\]
Here we need to assume that the initial metric $g(0)$ has nonnegative scalar curvature.
It is well known that this property is preserved by the Ricci flow.

\begin{proof}[Proof of Theorem \ref{main3}]
This time $u$ satisfies
\[
\frac{\partial}{\partial \tau}u=\Delta u-|\nabla u|^2+R-u.
\]
Adapting  the equation (3.2) of \cite{[Caox]}, we can write
\begin{equation}
\begin{aligned}\label{conjevol1}
\frac{\partial}{\partial \tau}H&=\Delta H-2\nabla H\cdot \nabla
u-\frac{2}{\tau}H-\frac{2}{\tau}|\nabla u|^2-2\frac{R}{\tau}\\
&\quad-2\left|\nabla_i\nabla_ju+R_{ij}-\frac{1}{\tau}g_{ij}\right|^2
-2(\Delta u-|\nabla u|^2).
\end{aligned}
\end{equation}
Note that
\[
H=2\Delta u-|\nabla u|^2+R-2\frac n\tau,
\]
which implies
\[
-2(\Delta u-|\nabla u|^2)=-2H+2\left(\Delta
u+R-\frac{n}{\tau}\right)-\frac{2n}{\tau}.
\]
Plugging this into \eqref{conjevol1}, we obtain
\begin{equation*}
\begin{aligned}
\frac{\partial}{\partial\tau}H&\leq\Delta H-2\nabla H\cdot\nabla
u-\left(\frac{2}{\tau}+2\right)H-\frac{2}{\tau}|\nabla u|^2-2\frac{R}{\tau}\\
&\quad-\frac 2n\left(\Delta
u+R-\frac{n}{\tau}\right)^2+2\left(\Delta
u+R-\frac{n}{\tau}\right)-\frac{2n}{\tau}\\
&=\Delta H-2\nabla H\cdot \nabla
u-\left(\frac{2}{\tau}+2\right)H-\frac{2}{\tau}|\nabla u|^2-2\frac{R}{\tau}\\
&\quad-\frac 2n\left(\Delta u+R-\frac{n}{\tau}-\frac
n2\right)^2-\frac{2n}{\tau}+\frac n2.
\end{aligned}
\end{equation*}
Adding $-\frac n4$ to $H$ yields
\begin{equation}
\begin{aligned}\label{conjevol3}
\frac{\partial}{\partial \tau}\left(H-\frac n4\right)&\leq\Delta
\left(H-\frac n4\right)-2\nabla\left(H-\frac n4\right)\cdot\nabla
u-\left(\frac{2}{\tau}+2\right)\left(H-\frac n4\right)\\
&\quad-\frac{2}{\tau}|\nabla u|^2-2\frac{R}{\tau}-\frac
2n\left(\Delta u+R-\frac{n}{\tau}-\frac
n2\right)^2-\frac{5n}{2\tau}.
\end{aligned}
\end{equation}
Since $R\geq 0$, it is easy to see that $H-\frac n4<0$ for $\tau$ small enough.
Applying the maximum principle to the evolution formula \eqref{conjevol3}, we have
\[
H-\frac n4\leq 0
\]
for all time $\tau$, hence for all $t$.  This finishes the proof of
Theorem \ref{main3}.
\end{proof}

We easily derive counterparts to Theorem \ref{main2} and Theorem \ref{classmain1}:

\begin{theorem}\label{maintyp2}
Let $(M, g(t))$, $t\in[0,T)$ (where $T<\infty$ is the blow-up time) be a
solution to the Ricci flow on a closed manifold of dimension $n$ with nonnegative
scalar curvature, and assume that $g$ is of type I, that is, it satisfies
\eqref{type1}, for some constant $d_0$. Let $f$ be a positive solution to
the nonlinear backward heat equation \eqref{conjeq1}, $u=-\ln f$, $\tau=T-t$
and
\[
H=2\Delta u-|\nabla u|^2+R-d\frac{n}{\tau},
\]
where $d=d(d_0,n)\geq 1$ is some constant such that $H(\tau)<0$ for
small $\tau$. Then, for all time $t\in[0,T)$,
\[
H\leq \frac n4.
\]
\end{theorem}

\begin{theorem}\label{classmain3}
Let $(M,g(t))$, $t\in[0,T]$, be a solution to the Ricci flow on a closed
manifold of dimension $n$ with nonnegative scalar curvature. Let $f$ be a
positive solution to the nonlinear backward heat equation
\eqref{conjeq1}. Assume that $(x_1,t_1)$ and $(x_2,t_2)$, $0\leq
t_1<t_2<T$, are two points in $M\times[0,T)$. Then
\[
e^{t_2}\cdot\ln f(x_2,t_2)-e^{t_1}\cdot\ln f(x_1,t_1)\leq\frac
12\int^{t_2}_{t_1}e^{T-t}\left(|\dot{\gamma}|^2+R+\frac
n4+\frac{2n}{T-t}\right)dt,
\]
where $\gamma$ is any space-time path joining $(x_1,t_1)$ and $(x_2,
t_2)$.
\end{theorem}

\vspace{0.5em}

In the rest of this section, we will finish the proof of Theorem
\ref{main3b}. The interesting feature of Theorem \ref{main3b}
is that the differential Harnack inequalities
hold \emph{without} any assumption on the scalar curvature $R$.
\begin{proof}[Proof of Theorem \ref{main3b}]
We first compute that $v$ satisfies
\begin{equation}\label{conjeq2b}
\frac{\partial}{\partial \tau}v=\Delta v-|\nabla
v|^2+R-\frac{n}{2\tau}-\left(v+\frac n2\ln(4\pi\tau)\right).
\end{equation}
If we let
\[
\tilde{P}:=2\Delta v-|\nabla v|^2+R-2\frac n\tau,
\]
then by adapting  equation (3.7) in \cite{[Caox]}, we have
\begin{equation*}
\begin{aligned}
\frac{\partial}{\partial \tau}\tilde{P}
&=\Delta \tilde{P}-2\nabla \tilde{P}\cdot \nabla
v-\frac{2}{\tau}\tilde{P}-\frac{2}{\tau}|\nabla v|^2-2\frac{R}{\tau}\\
&\quad-2\left|\nabla_i\nabla_jv+R_{ij}-\frac{1}{\tau}g_{ij}\right|^2
-2(\Delta v-|\nabla v|^2).
\end{aligned}
\end{equation*}
Since $P=\tilde{P}-\frac{n}{\tau}$, we have
\begin{equation}
\begin{aligned}\label{conjl1bch}
\frac{\partial}{\partial \tau}P&=\Delta P-2\nabla P\cdot \nabla
v-\frac{2}{\tau}P-\frac{2}{\tau}|\nabla v|^2-2\frac{R}{\tau}-\frac{n}{\tau^2}\\
&\quad-2\left|\nabla_i\nabla_jv+R_{ij}-\frac{1}{\tau}g_{ij}\right|^2
-2(\Delta v-|\nabla v|^2).
\end{aligned}
\end{equation}
According to the definition of $P$, we have
\[
-2a(\Delta v-|\nabla v|^2)=-2P+2\left(\Delta
v+R-\frac{n}{\tau}\right)-\frac{4n}{\tau}.
\]
Substituting this into \eqref{conjl1bch}, we get
\begin{equation}
\begin{aligned}\label{Rchange}
\frac{\partial}{\partial\tau}P&\leq\Delta P-2\nabla P\cdot\nabla
v-\left(\frac{2}{\tau}+2\right)P-\frac{2}{\tau}|\nabla v|^2
-2\frac{R}{\tau}-\frac{n}{\tau^2}\\
&\quad-\frac 2n\left(\Delta v+R-\frac{n}{\tau}\right)^2
+2\left(\Delta v+R-\frac{n}{\tau}\right)-\frac{4n}{\tau}\\
&=\Delta P-2\nabla P\cdot \nabla
v-\left(\frac{2}{\tau}+2\right)P-\frac{2}{\tau}|\nabla v|^2
-\frac{2}{\tau}\left(R+\frac{n}{2\tau}\right)\\
&\quad-\frac 2n\left(\Delta v+R-\frac{n}{\tau}-\frac
n2\right)^2-\frac{4n}{\tau}+\frac n2.
\end{aligned}
\end{equation}
Note that the evolution of scalar curvature under the Ricci flow is
\[
\frac{\partial R}{\partial t}=\Delta R+2|Ric|^2\geq \Delta R+\frac
2n R^2.
\]
Applying the maximum principle to this inequality yields
\[
R\geq -\frac{n}{2t}.
\]
Since $t\geq T/2$, then $1/t\leq 1/\tau$. Hence
\[
R\geq -\frac{n}{2t}\geq -\frac{n}{2\tau},
\]
that is,
\[
R+\frac{n}{2\tau}\geq 0.
\]
Combining this with \eqref{Rchange}, we have
\[
\frac{\partial}{\partial\tau}P\leq\Delta P-2\nabla
P\cdot \nabla v-\left(\frac{2}{\tau}+2\right)P
-\frac{4n}{\tau}+\frac n2.
\]
Adding $-\frac n4$ to $P$, we get
\begin{equation}
\begin{aligned}\label{conjevolch}
\frac{\partial}{\partial \tau}\left(P-\frac n4\right)
&\leq\Delta \left(P-\frac n4\right)
-2\nabla\left(P-\frac n4\right)\cdot\nabla v
-\left(\frac{2}{\tau}+2\right)\left(P-\frac
n4\right)-\frac{9n}{2\tau}.
\end{aligned}
\end{equation}
It is easy to see that $P-\frac n4<0$ for $\tau$ small enough.
Applying the maximum principle to the
evolution formula \eqref{conjevolch} yields
\[
P-\frac n4\leq 0
\]
for all time $t\geq T/2$. Hence the theorem is completely proved.
\end{proof}

\begin{remark}
Motivated by Theorems \ref{maintyp2} and \ref{classmain3}, we can
prove similar theorems by the standard argument from Theorem
\ref{main3b}. We omit them in the interests of brevity.
\end{remark}

% --------------------------------------------------------------------------
% ------------------------------------------------------------------------

\section{Gradient estimates for nonlinear (backward) heat equations}\label{sec5}

In this section, on one hand we consider the positive solution $f(x,t)<1$
to the nonlinear heat equation without any potential
\begin{equation}\label{nheat1}
\frac{\partial}{\partial t}f=\Delta f-f\ln f,
\end{equation}
with the metric evolved by the Ricci flow \eqref{RF} on a closed
manifold $M$. This equation has been considered by
S.-Y. Hsu \cite{[Hsu]} and L. Ma~\cite{[Ma2]}.
If we let $u=-\ln f$, then
\begin{equation}\label{ufach}
\frac{\partial}{\partial t}u=\Delta u-|\nabla u|^2-u
\end{equation}
and $u>0$. Note that $0<f<1$ is preserved as time $t$ evolves.
In fact the initial assumption says that
\[
-\ln\sup_M f(x,0)\leq u(x,0)\leq-\ln\inf_M f(x,0).
\]
Applying the maximum principle to \eqref{ufach}, we have
\[
-\ln\sup_M f(x,0)e^{-t}\leq u(x,t)\leq-\ln\inf_M f(x,0)e^{-t}
\]
and hence
\[
0<u(x,t)\leq-\ln\inf_M f(x,0)
\]
for all $x\in M$ and $t\in[0,T)$. Since $u=-\ln f$, this implies
\[
0<\inf_M f(x,0)\leq f(x,t)<1
\]
for all $x\in M$ and $t\in[0,T)$.

Following the arguments of X. Cao and R. Hamilton's paper
\cite{[CaoxHa]}, we let
\[
H=|\nabla u|^2-\frac ut.
\]
Comparing with the equation (5.3) in the same reference, we have
\begin{equation}
\begin{aligned}\label{kkehear}
\frac{\partial}{\partial t}H&=\Delta H-2\nabla H\cdot \nabla
u-\frac1tH-2|\nabla\nabla u|^2-2|\nabla u|^2+\frac ut \\
&=\Delta H-2\nabla H\cdot \nabla
u-\left(\frac1t+1\right)H-2|\nabla\nabla u|^2-|\nabla u|^2.
\end{aligned}
\end{equation}
Notice that if $t$ small enough, then $H<0$. Then applying the
maximum principle to \eqref{kkehear}, we obtain:
\begin{theorem}\label{heatmain1}
Let $(M,g(t))$, $t\in[0,T)$, be a solution to the Ricci flow on a
closed manifold. Let $f<1$ be a positive solution to the nonlinear
heat equation \eqref{nheat1}, $u=-\ln f$ and
\[
H=|\nabla u|^2-\frac ut.
\]
Then, for all time $t\in(0,T)$,
\[
H\leq 0.
\]
\end{theorem}

\begin{remark}
Theorem \ref{heatmain1} can be regarded as a nonlinear version of
X. Cao and R. Hamilton's result (see Theorem 5.1 in \cite{[CaoxHa]}).
Recently, L. Ma has proved the same estimate as in Theorem
\ref{heatmain1} on a closed manifold with nonnegative Ricci curvature under
the static metric (see Theorem 3 in~\cite{[Ma2]}). However, in our case,
we do not need any curvature assumption.
\end{remark}

\vspace{0.5em}

On the other hand, we can also consider the positive solution
$f(x,t)<1$ to the nonlinear backward heat equation without any
potential
\begin{equation}\label{nheat2}
\frac{\partial}{\partial t}f=-\Delta f+f\ln f,
\end{equation}
with the metric evolved by the Ricci flow \eqref{RF}.
Let $u=-\ln f$. Then we have
\[
\frac{\partial}{\partial \tau}u=\Delta u-|\nabla u|^2-u
\]
and $u>0$. Using the maximum principle, one can see that
$0<f<1$ is also preserved under the Ricci flow.
In fact from the initial assumption
\[
0<\inf_M f(x,T)\leq f(x,T)\leq\sup_M f(x,T)<1,
\]
one can also show that
\[
0<\inf_M f(x,T)\leq f(x,\tau)<1
\]
for all $x\in M$ and $\tau\in(0,T]$ in the same way as the above arguments.

Following the arguments of X. Cao's paper \cite{[Caox]}, let
\[
H=|\nabla u|^2-\frac u\tau.
\]
Comparing with the equation (5.3) in \cite{[Caox]}, we have
\begin{equation}
\begin{aligned}\label{kkeheard}
\frac{\partial}{\partial \tau}H&=\Delta H-2\nabla H\cdot \nabla
u-\frac1\tau H-2|\nabla\nabla u|^2-4R_{ij}u_iu_j-2|\nabla
u|^2+\frac u\tau \\
&=\Delta H-2\nabla H\cdot \nabla u-\left(\frac1\tau+1\right)
H-2|\nabla\nabla u|^2-4R_{ij}u_iu_j-|\nabla u|^2.
\end{aligned}
\end{equation}
If we assume $R_{ij}(g(t))\geq -K$, where $0\leq K\leq \frac 14$,
then
\[
-4R_{ij}u_iu_j-|\nabla u|^2\leq (4K-1)|\nabla u|^2\leq0.
\]
Hence if $\tau$ small enough, then $H<0$. Then applying the maximum
principle to \eqref{kkeheard}, we have a nonlinear version of
X. Cao's result (see Theorem 5.1 in \cite{[Caox]}).

\begin{theorem}\label{heatmain2}
Let $(M,g(t))$, $t\in[0,T]$, be a solution to the Ricci flow on a
closed manifold with the Ricci curvature satisfying
$R_{ij}(g(t))\geq -K$, where $0\leq K\leq \frac 14$. Let $f<1$ be a
positive solution to the nonlinear backward heat equation
\eqref{nheat2}, $u=-\ln f$, $\tau=T-t$ and
\[
H=|\nabla u|^2-\frac u\tau.
\]
Then, for all time $t\in[0,T)$,
\[
H\leq 0.
\]
\end{theorem}

% ------------------------------------------------------------------------
\section*{Acknowledgments}
The author would like to express his gratitude to the referee for careful
readings and many valuable suggestions.

% --------------------------------------------------------------------------


\begin{thebibliography}{99}
\bibitem{[Andrews]}B. Andrews, Harnack inequalities for evolving hypersurfaces,
Math. Zeit., 217 (1994), 179-197.

\bibitem{[ArBe]} D.G. Aronson, P. B\'enilan, R\'egularit\'e des solutions de
l'\'equation des milieux poreux dans $\mathbb{R}^n$, C. R. Acad. Sci. Paris.
S\'er. A-B 288 (1979), A103-A105.

\bibitem{[BaCaoPu]} M. Bailesteanua, X.-D. Cao, A. Pulemotov, Gradient
estimates for the heat equation under the Ricci flow, J. Funct.
Anal., 258 (2010), 3517-3542.

\bibitem{[Cao]}H.-D. Cao, On Harnack's inequalities for the K\"ahler-Ricci flow,
Invent. Math., 109 (1993), 247-263.

\bibitem{[CaoNi]}H.-D. Cao, L. Ni, Matrix Li-Yau-Hamilton estimates for the
heat equation on K\"ahler manifolds, Math. Ann., 331 (2005),
795-807.

\bibitem{[Caox]}X.-D. Cao, Differential Harnack estimates for backward heat
equations with potentials under the Ricci flow, J. Funct. Anal., 255
(2008), 1024-1038.

\bibitem{[CaoxHa]}X.-D. Cao, R. S. Hamilton, Differential Harnack estimates for
time-dependent heat equations with potentials, Geom. Funct. Anal.,
19 (2009), 989-1000.

\bibitem{[CaoZhang]}X.-D. Cao, Z. Zhang, Differential Harnack estimates for
parabolic equations, Proceedings of Complex and Differential Geometry, 87-98, 2011.

\bibitem{[CaoZhq]}X.-D. Cao, Qi S. Zhang, The conjugate heat equation and
ancient solutions of the Ricci flow, Adv. Math., 228 (5) (2011), 2891-2919.

\bibitem{[ChTaYu]} A. Chau, L.-F. Tam, C.-J. Yu, Pseudolocality for the
Ricci flow and applications, Canad. J. Math., 63 (2011), 55-85.

\bibitem{[HBCheng]}H.-B. Cheng, A new Li-Yau-Hamilton estimate for the Ricci
flow, Comm. Anal. Geom., 14 (2006), 551-564.

\bibitem{[ChCh]}L. Chen, W.-Y. Chen, Gradient estimates for a nonlinear parabolic
equation on complete non-compact Riemannian manifolds, Ann. Global
Anal. Geom., 35 (2009), 397-404.

\bibitem{[Chow0]}B. Chow, The Ricci flow on the 2-sphere, J. Differential Geom.,
 33 (1991), 325-334. 

\bibitem{[Chow1]}B. Chow, On Harnack's inequality and entropy for the Gaussian
curvature flow, Comm. Pure Appl. Math., 44 (1991), 469-483.

\bibitem{[Chow2]}B. Chow, The Yamabe flow on locally conformally at manifolds
with positive Ricci curvature, Comm. Pure Appl. Math., 45 (1992),
1003-1014.

\bibitem{[Chow3]}B. Chow, Interpolating between Li-Yau's and Hamilton's
Harnack inequalities on a surface, J. Partial Differ. Equ.,
11 (1998), 137-140.

\bibitem{CCG3}B. Chow, S. C. Chu, D. Glickenstein,  C. Guentheretc,
J. Isenberg, T. Ivey, D. Knopf, P. Lu, F. Luo, L. Ni,  The Ricci
Flow: Techniques and Applications. Part III. Geometric-Analytic
Aspects. Mathematical Surveys and Monographs 163, American
Mathematical Society, Providence, RI, 2010.

\bibitem{[ChHam]}B. Chow, R. Hamilton, Constrained and linear Harnack inqualities
for parabolic equations, Invent. Math., 129 (1997), 213-238.

\bibitem{[ChKn]}B. Chow, D. Knopf, New Li-Yau-Hamilton inequalities for the
Ricci flow via the space-time approach, J. Diff. Geom., 60 (2002),
1-54.

\bibitem{[CLN]}B. Chow, P. Lu, L. Ni, Hamilton's Ricci flow, Lectures in
Contemporary Mathematics 3, Science Press and American Mathematical
Society, 2006.

\bibitem{[ChuYau]}F. R. K. Chung, S.-T. Yau, Logarithmic Harnack
inequalities, Math. Res. Lett., 3 (1996), 793-812.

\bibitem{[Guen]} C. M. Guenther, The fundamental solution on manifolds with
time-dependent metrics, J. Geom. Anal., 12 (2002), 425-436.

\bibitem{[Ham1]}R. S. Hamilton, The Ricci flow on surfaces, Contemp. Math. 71
(1988), 237-262, Amer. Math. Soc., Providence, RI.

\bibitem{[Ham2]}R. S. Hamilton, The Harnack estimate for the Ricci flow,
J. Diff. Geom., 37 (1993), 225-243.

\bibitem{[Ham3]}R. S. Hamilton, A matrix Harnack estimate for the heat equation,
Comm. Anal. Geom., 1 (1993), 113-126.

\bibitem{[Ham4]}R. S. Hamilton, The Harnack estimate for the mean curvature flow,
J. Diff. Geom., 41 (1995), 215-226.

\bibitem{[Hsu]}S.-Y. Hsu, Gradient estimates for a nonlinear parabolic equation
under Ricci flow, 24 (2011), Differential Integral Equations, 645-652.

\bibitem{[HuangMa]}G.-Y. Huang, B.-Q. Ma, Gradient estimates for a nonlinear
parabolic equation on Riemannian manifolds, Arch. Math., 94 (2010),
265-275.

\bibitem{[KuZh]}S.-L. Kuang, Qi S. Zhang, A gradient estimate for all positive
solutions of the conjugate heat equation under Ricci flow, J. Funct.
Anal., 255 (2008), 1008-1023.

\bibitem{[LiXu]}J.-F. Li, X.-J. Xu, Differential Harnack inequalities on Riemannian
manifolds I: Linear heat equation, Adv. Math., 226 (5) (2011), 4456-4491.

\bibitem{[Li-Yau]}P. Li, S.-T. Yau, On the parabolic kernel of the Schrodinger
operator, Acta Math., 156 (1986), 153-201.

\bibitem{[Liu]}S.-P. Liu, Gradient estimate for solutions of the heat equation
under Ricci flow, Pacific J. Math., 243 (2009) 165-180.

\bibitem{[Ma]}L. Ma, Gradient estimates for a simple elliptic equation on
non-compact Riemannian manifolds, J. Funct. Anal., 241 (2006),
374-382.

\bibitem{[Ma2]}L. Ma, Hamilton type estimates for heat equations on
manifolds, (2010), arXiv: math.DG/1009.0603v1.

\bibitem{[Ma3]}L. Ma, Gradient estimates for a simple nonlinear heat equation on
manifolds, (2010), arXiv: math.DG/1009.0604v1.

\bibitem{[Ni1]}L. Ni, A matrix Li-Yau-Hamilton estimate for K\"ahler-Ricci flow,
J. Diff. Geom., 75 (2007), 303-358.

\bibitem{[Perelman]}G. Perelman, The entropy formula for the Ricci flow and its
geometric applications, (2002), arXiv:math.DG/0211159v1.

\bibitem{[Wu1]}J.-Y. Wu, Gradient estimates for a nonlinear diffusion equation
on complete manifolds, J. Partial Differ. Equ., 23 (2010), 68-79.

\bibitem{[Wu2]}J.-Y. Wu, Li-Yau type estimates for a nonlinear parabolic equation on
complete manifolds, J. Math. Anal. Appl., 369 (2010) 400-407.

\bibitem{[WuZheng]} J.-Y. Wu, Y. Zheng, Interpolating between constrained Li-Yau
and Chow-Hamilton Harnack inequalities on a surface, Arch. Math., 94
(2010), 591-600.

\bibitem{[Yang]}Y.-Y. Yang, Gradient estimates for a nonlinear parabolic equation on
Riemannian manifold, Proc. Amer. Math. Soc., 136 (2008), 4095-4102.

\bibitem{[Yau1]}S.-T. Yau, On the Harnack inequalities of partial differential
equations, Comm. Anal. Geom., 2 (1994), 431-430.

\bibitem{[Yau2]}S.-T. Yau, Harnack inequality for non-self-adjoint evolution
equations, Math. Res. Lett., 2 (1995), 387-399.

\bibitem{[Zhang]} Qi S. Zhang, Some gradient estimates for the heat equation on
domains and for an equation by Perelman, Int. Math. Res. Not., Art.
ID 92314, pp 1-39, 2006.
\end{thebibliography}
\end{document}